\documentclass[10pt]{amsart}
\usepackage{amsmath,amscd}
\usepackage{amsbsy}
\usepackage{amssymb}
\usepackage{amscd,amsthm}
\usepackage[all,cmtip]{xy}
\usepackage[russian,english]{babel}
\usepackage[utf8]{inputenc}
\usepackage{xcolor}
\usepackage{dsfont}

\usepackage{hyperref}

\renewcommand{\epsilon}{\varepsilon}

\newtheorem{theorem}{Theorem}

\renewcommand{\ell}{x}

\newtheorem{thm}{Theorem}\numberwithin{thm}{section}
\newtheorem{lem}[thm]{Lemma}
\newtheorem{prop}[thm]{Proposition}

\newtheorem*{con2}{Conjecture}

\newtheorem*{que2}{Question}

\newtheorem*{rema2}{Remark}

\begin{document}
	\begin{center}
		\huge{Almost perfect inhomogeneous powers in arithmetic progression}\\[1cm] 
	\end{center}
	\begin{center}
		\large{Sa$\mathrm{\check{s}}$a Novakovi$\mathrm{\acute{c}}$}\\[0,5cm]
		{\small June 2026}\\[0,2cm]
	\end{center}
		{\small \textbf{Abstract}. 
			Let $S$ be a finite set of primes and write $\mathbb{Z}_S$ for the set of those non-zero integers whose prime divisors belong to $S$. 
			Hajdu proved that the abc conjecture implies that the number of terms of any arithmetic progression in $H_S=\{\eta x^l\mid \eta\in \mathbb{Z}_S, x,l\in \mathbb{Z},\ \textnormal{with}\ x>0 \ \textnormal{and} \ l\geq 2 \}$ is bounded. Moreover, if $k\geq 3$ and the exponents of the powers are all $\geq 4$, then the number of such progressions are finite. We consider other sets and prove  similar statements for these sets.}
		\begin{center}
			\tableofcontents
		\end{center}
		\section{Introduction}
		\noindent
	Arithmetic progressions consisting of almost perfect powers are widely studied in the literature, in particular in the homogeneous case. By almost perfect homogeneous we mean arithmetic progressions of the form $a_0x_0^l,...,a_{k-1}x_{k-1}^l$, with $a_i,x_i\in \mathbb{Z}$ and $0\leq i\leq k-1$, where $l\geq 2$ is fixed. Usually, the coefficients $a_i$ are restricted, for instance in the sense that $P(a_i)\leq k-1$. Here $P(m)$ denotes the largest prime factor of $|m|$. Hajdu \cite{HAY} considered arithmetic progressions consisting of so called almost perfect \emph{inhomogeneous} powers. Let $S=\{p_1,...,p_s\}$ be any set of primes with $p_1<p_2<\cdots <p_s$ and write $\mathbb{Z}_S$ for the set of those non-zero integers whose prime divisors belong to $S$. Put
	$$
	H_S=\{\eta x^l\mid \eta\in \mathbb{Z}_S, x,l\in \mathbb{Z},\ \textnormal{with}\ x>0 \ \textnormal{and} \ l\geq 2 \}.
	$$
	Furthermore, it is assumed that $\eta$ is $l$-th power free. Then Hajdu proved the following result.
	\begin{theorem}[\cite{HAY}, Theorem 1]
		Suppose the abc conjecture holds. Let $h_0,...,h_{k-1}$ by any non-constant arithmetic progression in $H_S$, with $h_i=\eta_ix_i^{l_i}\quad (0\leq i\leq k-1)$ such that $\mathrm{gcd}(h_0,h_1)\leq c_0$ for some positive constant $c_0$. Then we have $\mathrm{max}(k,l)<c_1$, where $l=\underset{0\leq i\leq k-1}{\mathrm{max}}l_i$. Moreover, the number of such progressions with $k\geq 3$ and $l_i\geq 4$ is bounded by some positive constant $c_2$.
	\end{theorem}
	At this point we want to mention that Theorem 1 is related to the equation
	\begin{equation}
	n(n+d)\cdots (n+(k-1)d)=bx^l
	\end{equation}
	in non-zero integers $n,d,b,x, k\geq 2,l\geq 2$ with $\mathrm{gcd}(n,d)=1$ and $P(b)\leq k$. This equation and several spezializations or variants have a very extensive literature. We don't want to give a complete list here but instead refer to the papers \cite{BGH}, \cite{GY}, \cite{GY2}, \cite{SS}, \cite{SHO}, \cite{SHO2} and \cite{TI} and the references therein. In this context we also refer to the introduction of \cite{BGHT}. We only mention one result due to Gy\H{o}ry, Hajdu and Saradha \cite{GY2} which may be relevant from our viewpoint. They showed that under the abc conjecture equation (1) has only finitely many solutions if $l\geq 4$ and $k\geq 3$. One can show that using equation (1) one can write
	$$
	n+id=a_ix_i^l \quad \textnormal{with} \quad P(a_i)\leq k-1.
	$$
	Thus Theorem 1 yields a kind of extension of the mentioned result in \cite{GY2}. In \cite{BGHT} Bruin, Gy\H{o}ry, Hajdu and Tengely dealt with the following problem:
	\begin{que2}
		For all $k\geq 3$ characterize the non-constant arithmetic progressions $(h_0,...,h_{k-1})$ with $\mathrm{gcd}(h_0,h_1)=1$ such that each $h_i=x_i^{l_i}$ for some $x_i\in \mathbb{Z}$ and $l_i\geq 2$.
	\end{que2}
	\noindent
	They showed that for $k\geq 4$ and $L\geq 2$, there are only finitely many $k$-term integral arithmetic progressions of the form stated in the problem with $2\leq l_i\leq L$. This result is a direct consequence of a more general result, namely:
	\begin{theorem}[\cite{BGHT}, Theorem 2]
		Let $L,k$ and $D$ be positive integers with $L\geq 2, k\geq3$, and let $S$ a finite set of primes. Then there are at most finitely many arithmetic progressions $(h_0,...,h_{k-1})$ satisfying the following conditions:
		\begin{itemize}
			\item [(i)] For $i=0,...,k-1$, there exists $x_i\in \mathbb{Z}, 2\leq l_i\leq L$ and $\eta_i\in\mathbb{Z}_S$ such that $h_i=\eta_ix_i^{l_i}$,
			\item[(ii)] $\mathrm{gcd}(h_0,h_1)\leq D$,
			\item[(iii)] either $k\geq 5$, or $k=4$ and $l_i\geq 3$ for some $i$, or $k=3$ and $\frac{1}{l_0}+\frac{1}{l_1}+\frac{1}{l_2}<1$.
		\end{itemize}
	\end{theorem}
	\noindent
For a non-zero integer $y$, let $N(y)$ be the \emph{algebraic radical}, namely $N(y)=\underset{p\mid y}{\prod}{p}$. In the present note we are interessted in characterizing non-constant arithmetic progressions in other sets then $H_S$. So instead of $H_S$, we consider the sets
$$
G_{\alpha}=\{\eta x^l\mid \eta\in \mathbb{Z}, x,l\in \mathbb{Z},\  \textnormal{with}\ x>0, l\geq 2 \ \textnormal{and} \ N(\eta)\leq \alpha\},
$$
where $\alpha>0$ is a fixed integer, 
$$
G_{\delta}=\{\eta x^l\mid \eta\in \mathbb{Z}, x,l\in \mathbb{Z},\ \textnormal{with}\ x>0, l\geq 2 \ \textnormal{and} \ N(\eta)\leq x^{\delta}\},
$$
where $\delta$ is a fixed positive number $<1/1763$,
$$
	G_{!}=\{l! x^l\mid x,l\in \mathbb{Z},\ \textnormal{with}\ x>0, l\geq 2 \}
$$
and
$$
G_f=\{f(l)x^l\mid x,l\in \mathbb{Z},\ \textnormal{with}\ x>0, l\geq 2 \},
$$
where $f:\mathbb{Z}\rightarrow \mathbb{Z}$ is a non-zero function. Moreover, we may assume $\eta$ to be $l$-th power free. It is easy to see that $G_{\alpha}\subset H_S$ for a suitable finite set $S$ of primes. Therefore, Theorem 3 is a direct consequence of Theorem 1. Throughout the note, we say a non-constant arithmetic progression $h_0,...,h_{k-1}$ is \emph{$l$-homogeneous} if $h_i=\eta_ix_i^l$ for a fixed $l\geq 2$ and $\eta_i\in \mathbb{Z}$, $\eta_i\neq 0$. If moreover any arithmetic progression in a family of $l$-homogeneous  progessions satisfies $k<c(l)$ for a positive constant $c(l)$, depending only on $l$, we say the family of $l$-homogeneous progressions is $l$-\emph{smooth}. For example the family of $l$-homogeneous progressions in the set $H_S$ is $l$-smooth for any $l\geq 2$, as proved by Hajdu (see \cite{HAY}, Lemma 4). Notice that the proof needs results from Ramsey theory and contributions of Darmon and Merel \cite{DM} on the Fermat equation. So this is highly non-trivial.
	Since we mentioned the abc conjecture and since we also need it to prove Theorems 4, 5 and 6 below, we recall its statement. 
		\begin{con2}[classical abc conjecture]
			For any $\epsilon >0$ there is a constant $K(\epsilon)$ depending only on $\epsilon$ such that whenever $a,b$ and $c$  are three coprime and non-zero integers with $a+b=c$, then 
			\begin{eqnarray*}
				c<K(\epsilon)N(abc)^{1+\epsilon}
			\end{eqnarray*}
			holds.
		\end{con2} 
	\begin{theorem}
		Suppose that the abc conjecture holds. Fix a positive integer $\alpha$ and let $h_0,...,h_{k-1}$ by any non-constant arithmetic progression in $G_{\alpha}$, with $h_i=\eta_ix_i^{l_i}\quad (0\leq i\leq k-1)$ such that $\mathrm{gcd}(h_0,h_1)\leq c_0$ for some positive constant $c_0$. Then we have $\mathrm{max}(k,l)<c_1$, where $l=\underset{0\leq i\leq k-1}{\mathrm{max}}l_i$. Moreover, the number of such progressions with $l_i\geq 4$ is bounded.
	\end{theorem}
		
		\begin{theorem}
			Suppose that the abc conjecture is valid and assume that $G_{\delta}$ is $2$-smooth with corresponding constant $c(2)$. Now let $h_0,...,h_{k-1}$ be any non-constant arithmetic progression in $G_{\delta}$, with $h_i=\eta_ix_i^{l_i}$ such that $\mathrm{gcd}(h_0,h_1)\leq c_3$ for some positive constant $c_3$ and $k\geq \mathrm{max}\{3,2c(2)\}$. Then we have $\mathrm{max}(k,l)<c_4$, where $l=\underset{0\leq i\leq k-1}{\mathrm{max}}l_i$ and the number of such progressions is bounded whenever $k\geq 3$ and $l_i\geq 4$.
		\end{theorem}
		\noindent
		We denote by $\omega(y)$ the number of distinct prime divisors of the integer $y$. For a non-constant arithmetic progression $h_0,...,h_{k-1}$, let $d=h_1-h_0$. 
		\begin{theorem}
			Suppose that the abc conjecture is valid. Let $h_0,...,h_{k-1}$ be any non-constant arithmetic progression in $G_{!}$, with $h_i=l_i!x_i^{l_i}$ such that $\mathrm{gcd}(h_0,h_1)\leq c_5$ for some positive constant $c_5$. Assume $\omega(d)$ is bounded. Then number of such progressions with $k\geq 3$ and $l_i\geq 4$ is bounded.
		\end{theorem}
		\begin{rema2}
			\textnormal{As for Theorem 1, the condition that $\mathrm{gcd}(h_0,h_1)$ is bounded in Theorems 3 and 4 is necessary. Otherwise, there exist non-constant arithmetic progressions in $G_{\alpha}$, and $G_{\delta}$ respectively, consisting of non-zero perfect powers having arbitrary many terms. See \cite{HAY}, Remark 2 for examples.}
		\end{rema2}
		\noindent
		The proof of Theorem 5 shows that we can also formulate the following, more general, result. Let $f:\mathbb{Z}\rightarrow \mathbb{Z}$ be a non-zero function and 
		$$
		G_f=\{f(l)x^l\mid x,l\in \mathbb{Z},\ \textnormal{with}\ x>0, l\geq 2\}.
		$$
		\begin{theorem}
			Let $f$ be a non-zero function such that $f(m_1)\mid f(m_2)$ whenever $m_1\leq m_2$. Assume that the abc conjecture holds and that there exists a $\gamma >0$, such that $N(f(n))^{1+\gamma}=o(f(n))$ as $n\rightarrow \infty$. Let $h_0,...,h_{k-1}$ be a non-constant arithmetic progression in $G_f$ such that $\mathrm{gcd}(h_0,h_1)\leq c_{16}$ and $\omega(d)$ is bounded. Then there are only finitely many such progressions with $k>\mathrm{max}\{2,P(f(2))\}$ and $l_i\geq 4$. 
		\end{theorem}
		\begin{theorem}
			Let $L,k, D$ bepositive integers with $L\geq 2$, $k\geq 3$. Then there are finitely many non-constant arithmetic progressions $h_0,...,h_{k-1}$ in $G_f$ (or $G_{!}$ resp.), satisfying the following conditions:
			\begin{itemize}
				\item [(i)] For $i=0,...,k-1$, there exists $x_i\in \mathbb{Z}$, $2\leq l_i\leq L$ such that $h_i=f(l_i)x_i^{l_i}$ (or $h_i=l_i!x_i^{l_i}$ resp.),
				\item [(ii)] $\mathrm{gcd}(h_0,h_1)\leq D$,
				\item [(iii)] either $k\geq 5$, or $k=4$ and $l_i\geq 3$ for some $i$, or $k=3$ and $\frac{1}{l_0}+\frac{1}{l_1}+\frac{1}{l_2}<1$.
			\end{itemize}
		\end{theorem}
	
		\section{Preliminaries}
		\noindent
		We state some results that we need to prove Theorem 4 and 5. 
		\begin{lem}
			For every positive integers $u$ and $v$ there exists a positive integer $w$ such that for any coloring of the set $\{1,...,w\}$ using $u$ colors, we get a non-constant monochromatic arithmetic progression, having at least $v$ terms.
		\end{lem}
		\begin{proof}
			This is due to van der Waerden \cite{VW}.
		\end{proof}
		\begin{lem}
			Let $h_0,...,h_{k-1}$ be a $2$-homogeneous arithmetic progression in $G_{!}$. Assume $k\geq 3$ and that $\omega(d)$ is bounded. Then $k<c(2)$, where $c(2)$ is a positive constant, depending only on $l=2$.
		\end{lem}
		\begin{proof}
			Consider $\Pi=h_0\cdots h_{k-1}=2^kx^2$, where $x=x_0\cdots x_{k-1}$. Now $P(2^k)=2<k$. And we know from a theorem of Shory and Tijdeman \cite{SHT} that $k$ is bounded by a positive constant depending only on $\omega(d)$. Since $\omega(d)<a'$, $k<c(a',2)$. Essentially, $k<c(2)$. 
		\end{proof}
	\noindent
	In the context of the present paper, we also have the following easy consequence of Lemma 2.1.
		\begin{prop}
		Let $L\geq 2$ be a fixed integer and assume that any non-constant $l$-homogeneous arithmetic progression with $2\leq l\leq L$, is smooth. Then for any non-constant arithmetic progression $h_0,...,h_{k-1}$ with $h_ii=\eta_ix_i^{l_i}$ such that $l_i\leq L$, we have $k\leq C_0$ for a suitable constant $C_0$, depending only on $L$. 
		\end{prop}
		\begin{proof}
			For any $l$-homogeneous smooth arithmetic progression, we obtain a constant $c(l)$. Let $c_2(L)$ be the maximun of the values $c(l)$, where $l$ ranges through the interval $[2,L]$. Apply Lemma 2.1 to the progression $h_0,...,h_{k-1}$ with $(u,v)=(L-1,c_2(L))$. Here terms having the same exponent have the same color. Hence Lemma 2.1 gives some constant $C_0$, depending only on $L$, such that $k\geq C_0$. But this is a contradiction.
		\end{proof}
		\begin{thm}[see \cite{DG}]
			let $A,B,C$ and $r,s,t$ be non-zero integers with $r,s,t\geq 2$, and let $S$ be a finite set of primes. Then there exists a number field $K$ such that all solutions $x,y,z\in \mathbb{Z}$ with $\mathrm{gcd}(x,y,z)\in \mathbb{Z}_S$ to the equation
			$$
			Ax^r+By^s=Cz^t
			$$
			correspond, up to weighted projective equivalence, to $K$-rational points on some algebraic curve $X_{r,s,t}$ defined over $K$. Putting $u=-Ax^r/Cz^t$, the curve is a Galois-cover of the $u$-line of degree $d$, unramified outside $u\in \{0,1,\infty\}$ and with ramification indices $e_0=r, e_1=s$ and $e_2=t$. Writing $\chi=1/r+1/s+1/t$ and $g$ for the genus of the curve, we find
			\begin{itemize}
				\item [(i)] if $\chi>1$, then $g=0$ and $d=2/\chi$,
				\item [(ii)] if $\chi=1$, then $g=1$,
				\item [(iii)] if $\chi<1$, then $g>1$.
			\end{itemize}
		\end{thm}
		
		\section{Proof of Theorem 4}
		\noindent
		The theorem follows from the following lemma and from Theorem 1.
			\begin{lem}
			Let $G_{\delta}$ be 2-smooth and let $c(2)$ be the constant corresponding to a $2$-homogeneous smooth arithmetic progression in $G_{\delta}$. Suppose the abc conjecture holds. Let $h_0,...,h_{k-1}$ be any arithmetic progression in $G_{\delta}$ with $h_i=\eta_ix_i^{l_i}$ and $l_i\geq 4$ such that $\mathrm{gcd}(h_0,h_1)<c'_0$ and $k\geq \mathrm{max}\{3,2c(2)\}$, Then there exists a $c_7$ such $N(\eta_i)<c_7$ for all $i=0,...,k-1$.
		\end{lem}
		\begin{proof}
			We prove the statement in a slightly weaker form. Suppose we are given an arithmetic progression $h_0,...,h_{k-1}$ and that there is an $i\in \{0,...,k-1\}$ with $l_i\geq 7$. 
			Now there exists a $j$ with $0<|i-j|\leq c(2)$ such that $l_j\geq 3$. Because otherwise, there would be an $2$-homogeneous arithmetic progression of length $c(2)$ which is impossible. Choose any $t\in \{0,...,k-1\}\setminus \{i,j\}$ with $|i-t|\leq 2$. Then with some coprime non-zero integers $\lambda_i,\lambda_j, \lambda_t$ with $\mathrm{max}\{\lambda_i,\lambda_j,\lambda_t\}\leq |i-j|+2$ we have $\lambda_ih_i+\lambda_jh_j+\lambda_th_t=0$. But this gives us
			$$
			\lambda_i\eta_ix_i^{l_i}+\lambda_j\eta_jx_j^{l_j}+\lambda_t\eta_tx_t^{l_t}=0.
			$$
			Let $D$ be the greatest common divisor of the above terms. Since $\mathrm{gcd}(h_0,h_1)\leq c'_0$, we conclude $D$ is bounded. In what follows, we show that the $l_i$ are bounded. So let $r\in \{i,j,t\}$ be the index for which $|\lambda_r\eta_rx_r^{l_r}|$ is maximal among these three terms. Now apply the abc conjecture with $\epsilon=1/42$ to $a=\lambda_i\eta_ix_i^{l_i}$, $b=\lambda_j\eta_jx_j^{l_j}$ and $c=\lambda_t\eta_tx_t^{l_t}$. This yields
			$$
			|\lambda_r\eta_rx_r^{l_r}|<c_5N(\lambda_i\eta_ix_i^{l_i}\cdot\lambda_j\eta_jx_j^{l_j}\cdot \lambda_t\eta_tx_t^{l_t})^{43/42}.
			$$
			Since $\mathrm{max}\{\lambda_i,\lambda_j,\lambda_t\}\leq c_3+2$, we find
			$$
			|\lambda_r\eta_rx_r^{l_r}|<c_6N(\eta_ix_i^{l_i}\cdot\eta_jx_j^{l_j}\cdot\eta_tx_t^{l_t})^{43/42}.
			$$
			Now 
			$$
			N(\eta_ix_i^{l_i}\cdot\eta_jx_j^{l_j}\cdot\eta_tx_t^{l_t})\leq N(\eta_i)N(\eta_j)N(\eta_t)\cdot\underset{p\mid x_ix_jx_t}{\prod}p\leq N(\eta_i)N(\eta_j)N(\eta_t)\cdot x_ix_jx_t
			$$
			and since $N(\eta_i)\leq x_i^{\delta}$, we find
			$$
			N(\eta_ix_i^{l_i}\cdot\eta_jx_j^{l_j}\cdot\eta_tx_t^{l_t})\leq (x_ix_jx_t)^{1+\delta}.
			$$
			As $l_i\geq 7, l_j\geq 3$ and $l_t\geq 2$, we obtain $1/l_i+1/l_j+1/l_t<1-1/42$. Now we set $x^l=\mathrm{max}\{x_i^{l_i},x_j^{l_j}, x_t^{l_t}\}$. Then
			$$
			(x_ix_jx_t)^{(1+\delta)\frac{43}{42}}\leq x^{l(\frac{1}{l_i}+\frac{1}{l_j}+\frac{1}{l_t})(1+\delta)\frac{43}{42}}\leq x^{l\cdot\frac{1763}{1764}(1+\delta)}.
			$$
			This implies
			$$
			x^l\leq |\lambda_r\eta_rx_r^{l_r}|<c_6\cdot x^{l\cdot\frac{1763}{1764}(1+\delta)}.
			$$
			If $x^l=1$, then $x_i=x_j=x_t=1$ and therefore 
			$$
			\eta_i\leq |\lambda_i\eta_i|\leq |\lambda_r\eta_r|\leq c_6.
			$$
			This shows $N(\eta_i)<c_6$ if $x^l=1$. So we assume $x^l>1$.
			Since $\delta<1/1763$, we find that $x^l<c_7$. Hence $x_i^{l_i}<c_7$. This yields $N(\eta_i)<c_{7}$.\\
			\noindent
			\textbf{Proof of Theorem 4}\\
			\noindent
			From $N(\eta_i)<c_{7}$ it follows that there exists a suitable finite set $S$ of primes for which $G\subset H_S$. Now apply Theorem 1 to obtain the desired assertion. 
		\end{proof}
		\section{Proof of Theorems 5, 6 and 7}
		\noindent
		We prove Theorem 5. First we show that $l_i$ are bounded, using the abc conjecture. 
			\begin{lem}
			Let $c(2)$ be the constant corresponding to a $2$-homogeneous smooth arithmetic progression in $G_{!}$ from Lemma 2.2. Suppose the abc conjecture holds. Then there exists a $c_{11}$ such that if $h_0,...,h_{k-1}$ is any arithmetic progression in $G_{!}$ with $h_i=l_i!x_i^{l_i}$ such that $\mathrm{gcd}(h_0,h_1)<c_8$ and $k\geq 2c(2)$, then $l_i<c_{11}$ for all $i=0,...,k-1$.
		\end{lem} 
		\begin{proof}
				Suppose we are given an arithmetic progression $h_0,...,h_{k-1}$ and take any $i\in \{0,...,k-1\}$ with $l_i\geq 7$. Notice that if no such $i$ exists, the lemma follows with $c_{11}=7$. 
				Now there exists a $j$ with $0<|i-j|\leq c(2)$ such that $l_j\geq 3$. Because otherwise, there would be an $2$-homogeneous arithmetic progression of length $c(2)$ which is impossible. Choose any $t\in \{0,...,k-1\}\setminus \{i,j\}$ with $|i-t|\leq 2$. Then with some coprime non-zero integers $\lambda_i,\lambda_j, \lambda_t$ with $\mathrm{max}\{\lambda_i,\lambda_j,\lambda_t\}\leq |i-j|+2$ we have $\lambda_ih_i+\lambda_jh_j+\lambda_th_t=0$. But this gives us
					$$
				\lambda_il_i!x_i^{l_i}+\lambda_jl_j!x_j^{l_j}+\lambda_tl_t!x_t^{l_t}=0.
				$$
				Let $D$ be the greatest common divisor of the above terms. Since $\mathrm{gcd}(h_0,h_1)\leq c_8$, we conclude $D$ is bounded. So let $r\in \{i,j,t\}$ be the index for which $|\lambda_rl_r!x_r^{l_r}|$ is maximal among these three terms. Now apply the abc conjecture with to $a=\lambda_il_i!x_i^{l_i}$, $b=\lambda_jl_j!x_j^{l_j}$ and $c=\lambda_tl_t!x_t^{l_t}$. This yields
					$$
				|\lambda_rl_r!x_r^{l_r}|<c_9N(l_i!x_i^{l_i}\cdot l_j!x_j^{l_j}\cdot l_t!x_t^{l_t})^{1+\epsilon}.
				$$
				We set $H=|\lambda_rl_r!x_r^{l_r}|$ and let $l=\mathrm{max}\{l_i,l_j,l_t\}$. 
				Since $N(n!)\leq 4^n$, we obtain
				$$
				N(l_i!x_i^{l_i}\cdot l_j!x_j^{l_j}\cdot l_t!x_t^{l_t})\leq 4^{3l}\cdot x_ix_jx_t. 
				$$
				Notice that
				$$
				x_i\leq \left(\frac{H}{l_i!}\right)^{\frac{1}{l_i}}, \ x_j\leq \left(\frac{H}{l_j!}\right)^{\frac{1}{l_j}}, \ x_t\leq \left(\frac{H}{l_t!}\right)^{\frac{1}{l_t}}.
				$$
				And as $l_i\geq 7, l_j\geq 3$ and $l_t\geq 2$, we have
				$$
				\frac{1}{l_i}+\frac{1}{l_j}+\frac{1}{l_t}<1-\frac{1}{42},
				$$
				yielding
				$$
			H<c_94^{3l(1+\epsilon)}\cdot H^{(\frac{1}{l_i}+\frac{1}{l_j}+\frac{1}{l_t})(1+\epsilon)}<c_94^{3l(1+\epsilon)}\cdot H^{\frac{41}{42}(1+\epsilon)}
				$$
				Choosing for instance $\epsilon <1/50$ gives 
				$$
				H^{\delta}<c_{10}c'^l
				$$
				for a $\delta>0$ and a suitable positive constant $c'$. Now since $l!<H$, we find
				$$
				l!<c_{10}c'^l.
				$$
				Using $(l/e)^l<l!$, we find that $l$ must be bounded. 
		\end{proof}
		\noindent
		Now, since $l_i$ is bounded, we can use Theorem 2.4 to conclude that there are only finitely many solutions if $k\geq 3$ and $l_i\geq 4$. Since an arithmetic progression of length $>5$ contains an arithmetic progression of length $5$, we only have to consider the cases $k=3,4,5$. We know from Lemma 4.1 that $l_i< c_{11}$. Hence we only need to prove finiteness for a given exponent vector $(l_0,...,l_{k-1})$. Without loss of generality, $l_i$ can be taken to be $l_i$-th power free. Note that $\mathrm{gcd}(h_0,h_1)\leq c_8$, we also have $\mathrm{gcd}(x_i,x_j)\leq c_8$. So we can define $S$ to be the set of all primes up to $c_8$. We write $d=h_1-h_0$. The theorem will be proved if we show that the following system of equations has only finitely many solutions:
		\begin{itemize}
			\item [(a)] $l_i!x_i^{l_i}-l_j!x_j^{l_j}=(i-j)d$ for all $0\leq i<j\leq k-1$.
			\item [(b)] $(x_0,...,x_{k-1})\in \mathbb{Z}^k$ with $\mathrm{gcd}(x_0,x_1)\leq c_8$.
		\end{itemize}
		 So we need to solve
		 $$
		 (j-m)l_i!x_i^{l_i}+(m-i)l_jx_j^{l_j}+(i-j)l_m!x_m^{l_m}=0
		 $$
		 for all $0\leq m,i,j\leq k-1$. For $m=0, i=1$, we obtain that each of the solutions would give rise to a solution to the equation
		 $$
		 jl_1!x_1^{l_1}-l_j!x_j^{l_j}+(1-j)l_0!x_0^{l_0}=0.
		 $$
		 Now we apply Theorem 2.4 to this equation. According to Theorem 2.4, solutions give rise to $K_j$-rational points on some algebraic curve $C_j$ over some number field $K_j$. Putting
		 $$
		 u=\frac{l_1!x_1^{l_1}}{l_0!x_0^{l_0}},
		 $$
		 we obtain that $C_j$ is a Galois-cover of the $u$-line. with ramification indices $l_0,l_1,l_j$ over $u=\infty, 0, j/(j-1)$. If $k=3$, then $j=2$ and we see that Theorem 2.4 immediately implies that if $1/l_0+1/l_1+1/l_2<1$, the curve $C_2$ has genus larger than $1$. Thus by Falting's theorem, $C_2$ has only finitely many rational points.\\
		 \noindent
		 If $k=4$, we look at the solutions to 
		 $$
		 jl_1!x_1^{l_1}-l_j!x_j^{l_j}+(1-j)l_0!x_0^{l_0}=0
		 $$
		 for $j=2,3$ simultaneously. Let $M$ be a number field containing both $K_2$ and $K_3$. The solutions we are interessted in correspond to the rational points on the fibre product $C_2\times_{u} C_3$. This fibre product is Galois over the $u$-line and has ramification indices at least $l_0,l_1,l_2,l_3$ over $u=\infty, 0, 2, \frac{3}{2}$. All of the connected components have the same degree $d$. The Riemann-Hurwitz formula gives us for the genus of a connected component
		 $$
		 2g-2\geq d\left(2-\frac{1}{l_0}-\frac{1}{l_1}-\frac{1}{l_2}-\frac{1}{l_3}\right).
		 $$
		 We see that $g\geq 1$ only if $l_0=l_1=l_2=l_3=2$.  For all other cases, we find $g\geq 2$. And again by Falting's theorem there are only finitely many rational points in these cases. The case $k=5$ can be proved in a similar way. In this situation, we consider $C_2\times_uC_3\times_uC_4$ with ramification indices at lest $l_0,l_1,l_2,l_3,l_4$ over $u=0, \infty, 1, \frac{3}{2}, \frac{4}{3}$. In this case we allways get $g\geq 2$. \\
		 \noindent
		 \textbf{Proof of Theorem 6}
		 \noindent\\
		 Consider $\Pi=h_0\cdots h_{k-1}=f(2)^kx^2$, where $x=x_0\cdots x_{k-1}$. Now $P(f(2)^k)=P(f(2))<k$. And we know from a theorem of Shory and Tijdeman \cite{SHT} that $k$ is bounded by a positive constant depending only on $\omega(d)$. Since $\omega(d)<a''$, $k<c'(a'',2)$. Essentially, $k<c'(2)$. Now we prove the following lemma. Notice that it does not assume $l_i\geq 4$ for all $i$.
		 \begin{lem}
		 		Let $c'(2)$ be the constant corresponding to a $2$-homogeneous smooth arithmetic progression in $G_f$ from above. Suppose the abc conjecture holds. Then there exists a $c_{15}$ such that if $h_0,...,h_{k-1}$ is any arithmetic progression in $G_{f}$ with $h_i=f(l_i)!x_i^{l_i}$ such that $\mathrm{gcd}(h_0,h_1)<c_{12}$ and $k\geq 2c'(2)$, then $l_i<c_{15}$ for all $i=0,...,k-1$.
		 \end{lem}
		 \begin{proof}
		 	Suppose we are given an arithmetic progression $h_0,...,h_{k-1}$ and take any $i\in \{0,...,k-1\}$ with $l_i\geq 7$. Notice that if no such $i$ exists, the lemma follows with $c_{15}=7$. 
		 	Now there exists a $j$ with $0<|i-j|\leq c'(2)$ such that $l_j\geq 3$. Because otherwise, there would be an $2$-homogeneous arithmetic progression of length $c'(2)$ which is impossible. Choose any $t\in \{0,...,k-1\}\setminus \{i,j\}$ with $|i-t|\leq 2$. Then with some coprime non-zero integers $\lambda_i,\lambda_j, \lambda_t$ with $\mathrm{max}\{\lambda_i,\lambda_j,\lambda_t\}\leq |i-j|+2$ we have $\lambda_ih_i+\lambda_jh_j+\lambda_th_t=0$. But this gives us
		 	$$
		 	\lambda_if(l_i)x_i^{l_i}+\lambda_jf(l_j)x_j^{l_j}+\lambda_tf(l_t)x_t^{l_t}=0.
		 	$$
		 	Let $D$ be the greatest common divisor of the above terms. Since $\mathrm{gcd}(h_0,h_1)\leq c_{12}$, we conclude $D$ is bounded. So let $r\in \{i,j,t\}$ be the index for which $|\lambda_rf(l_r)x_r^{l_r}|$ is maximal among these three terms. Now apply the abc conjecture with to $a=\lambda_if(l_i)x_i^{l_i}$, $b=\lambda_jf(l_j)x_j^{l_j}$ and $c=\lambda_tf(l_t)x_t^{l_t}$. This yields
		 	$$
		 	|\lambda_rf(l_r)x_r^{l_r}|<c_{13}N(f(l_i)x_i^{l_i}\cdot f(l_j)x_j^{l_j}\cdot f(l_t)x_t^{l_t})^{1+\epsilon}.
		 	$$
		 	We set $H=|\lambda_rf(l_r)x_r^{l_r}|$ and let $l=\mathrm{max}\{l_i,l_j,l_t\}$. 
		 	From the assumption on the function $f$, we obtain
		 	$$
		 	N(f(l_i)x_i^{l_i}\cdot f(l_j)x_j^{l_j}\cdot f(l_t)x_t^{l_t})\leq N(f(l))\cdot x_ix_jx_t. 
		 	$$
		 	Notice that
		 	$$
		 	x_i\leq \left(\frac{H}{f(l_i)}\right)^{\frac{1}{l_i}}, \ x_j\leq \left(\frac{H}{f(l_j)}\right)^{\frac{1}{l_j}}, \ x_t\leq \left(\frac{H}{f(l_t)}\right)^{\frac{1}{l_t}}.
		 	$$
		 	And as $l_i\geq 7, l_j\geq 3$ and $l_t\geq 2$, we have
		 	$$
		 	\frac{1}{l_i}+\frac{1}{l_j}+\frac{1}{l_t}<1-\frac{1}{42},
		 	$$
		 	yielding
		 	$$
		 	H<c_9N(f(l))^{(1+\epsilon)}\cdot H^{(\frac{1}{l_i}+\frac{1}{l_j}+\frac{1}{l_t})(1+\epsilon)}<c_9N(f(l))^{(1+\epsilon)}\cdot H^{\frac{41}{42}(1+\epsilon)}
		 	$$
		 	Choosing for instance $\epsilon <1/50$ gives 
		 	$$
		 	H^{\delta}<c_{14}N(f(l))^{(1+\epsilon)}
		 	$$
		 	for a fixed $\delta>0$. Now since $f(l)<H$, we find
		 	$$
		 f(l)<c_{14}N(f(l))^{(1+\epsilon)}.
		 	$$
		 	Since 
		 	$$
		 	N(f(l))^{(1+\epsilon)}<N(f(l))^{(1+\gamma)},
		 	$$
		 	we conclude $N(f(l))^{1+\epsilon}=o(f(l))$ as $l\rightarrow \infty$, whenever $\epsilon<\gamma$. So we can choose $\epsilon$ such that $\epsilon <1/50$ and $\epsilon <\gamma$. This shows that $l$ must be bounded. 
		 \end{proof}
		 \noindent
		 To prove Theorem 6, use Theorem 2.4 and proceed as in the second part of the proof of Theorem 5.

		 \noindent
		 \textbf{Proof of Theorem 7}\\
		 \noindent
		 We sketch the argument only for $G_f$. Since $l_i$ is bounded by assumption, we can use Theorem 2.4 to conclude that there are only finitely many solutions if $k\geq 3$ and $l_i\geq 4$. Since an arithmetic progression of length $>5$ contains an arithmetic progression of length $5$, we only hav to consider the cases $k=3,4,5$. Hence we only need to prove finiteness for a given exponent vector $(l_0,...,l_{k-1})$. Without loss of generality, $f(l_i)$ can be taken to be $l_i$-th power free. Note that $\mathrm{gcd}(h_0,h_1)\leq D$, we also have $\mathrm{gcd}(x_i,x_j)\leq D$. So we can define $S$ to be the set of all primes up to $D$. We write $d=h_1-h_0$. The theorem will be proved if we show that the following system of equations has only finitely many solutions:
		 \begin{itemize}
		 	\item [(a)] $f(l_i)x_i^{l_i}-f(l_j)x_j^{l_j}=(i-j)d$ for all $0\leq i<j\leq k-1$.
		 	\item [(b)] $(x_0,...,x_{k-1})\in \mathbb{Z}^k$ with $\mathrm{gcd}(x_0,x_1)\leq D$.
		 \end{itemize}
		 So we need to solve
		 $$
		 (j-m)f(l_i)x_i^{l_i}+(m-i)f(l_j)x_j^{l_j}+(i-j)f(l_m)x_m^{l_m}=0
		 $$
		 for all $0\leq m,i,j\leq k-1$. For $m=0, i=1$, we obtain that each of the solutions would give rise to a solution to the equation
		 $$
		 jf(l_1)x_1^{l_1}-f(l_j)x_j^{l_j}+(1-j)f(l_0)x_0^{l_0}=0.
		 $$
		 Now we apply Theorem 2.5 to this equation and conclude finiteness in the desired cases in the same way as in the second part of the proof of Theorem 5.

		\vspace{0.3cm}
		\noindent
		{\tiny HOCHSCHULE FRESENIUS UNIVERSITY OF APPLIED SCIENCES 40476 D\"USSELDORF, GERMANY.}\\
		E-mail adress: sasa.novakovic@hs-fresenius.de\\
		

\begin{thebibliography}{999}
			\bibitem{BGH} M.A. Bennett, K. Gy\H{o}ry an L. Hajdu, Powers from priducts of censecutive terms in arithmetic progression. J. reine Angew. Math. 
			\bibitem{BGHT} N. Bruin, K. Gy\H{o}ry, L. Hajdu and Sz. Tengely, Arithmetic progressions consisting of unlike powers. Indag. Math. 18 (2006). 539-555. 
			\bibitem{DG} H. Darmon and A. Granville, On the equations $z^m=F(x,y)$ and $Ax^p+By^q=Cz^r$. Bull. London Math. Soc. 27 (1995), 882-885.
			\bibitem{DM} H. Darmon and L. Merel, Winding quotients and some variants of Fermat's Last Theorem. J. reine Angew. Math. 490 (1997), 81-100.
			\bibitem{GY} K. Gy\H{o}ry, Power values of products of consecutive integers and binomial coefficients, Number Theory and Its Applications. Kluwer Acad. Publ. (1999), 145-156.
			\bibitem{GY2}, K. Gy\H{o}ry, L. Hajdu and N. Saradha, On the diophantine equation $n(n+d)...(n+(k-1)d)=by^l$, Canad. Math. J. 
			\bibitem{HAY} L. Hajdu, Perfect powers in arithmetic progression. A note on the inhomogeneous case. Acta Arith. 113 (2004), 343-349.
			
			\bibitem{SS} N. Saradha and T.N. Shorey, Contributions towards a conjecture of Erd\H{o}s on perfect powers in arithmetic progression. J. reine Angew. Math.
			\bibitem{SHO} T.N. Shorey, Exponential diophantine equations invilving products of censecutive integers and related equations. Number Theory (R.P. Bambach, V.C. Dumir and R.J. Hans-Gill, eds), Hindustan Book Agency (1999), 463-495.
			\bibitem{SHO2} T.N. Shorey, Powers in arithmetic progression, A Panorama in Number Theory (G. W\"{u}stholz, ed.), Cambridge University Press, Cambridge (2002), 325-336.
			\bibitem{SHT} T.N. Shorey and R. Tijdeman, Perfect powers in products of terms in an arithmetical progression.Compositio Math. 75 (1990), 307-344.
			\bibitem{TI} R. Tijdeman, Diophantine equations and diophantine approximations, Number Theory and Applications (R.A. Mollin, ed.), Kluwer Acad. Press (1989), 215-243.
			\bibitem{VW} B.L. van der Waerden, Beweis einer Baudetschen Vermutung, Nieuw Archief voor Wiskunde 19 (1927), 212-216.
			
		\end{thebibliography}
	\end{document}